\documentclass[12pt]{amsart}
\usepackage{amssymb,latexsym, array, pictex}
\usepackage{verbatim,euscript}
\usepackage{fullpage,color}

\usepackage[colorlinks=true,linkcolor=blue,urlcolor=my_color,citecolor=magenta]{hyperref}

\definecolor{my_color}{rgb}{0,0.5,0.5}
\definecolor{MIXT}{rgb}{0.4,0.3,0.6}

\input {cyracc.def}
\numberwithin{equation}{section}

\newtheorem{thm}{Theorem}[section]
\newtheorem{lm}[thm]{Lemma} 

\newtheorem{prop}[thm]{Proposition}

\theoremstyle{remark}

\newtheorem{rmk}[thm]{Remark}

\theoremstyle{definition}
\newtheorem{ex}[thm]{Example}

\newtheorem*{rema}{Remark}

\newcommand {\be}{{\mathfrak b}}

\newcommand {\g}{{\mathfrak g}}

\newcommand {\el}{{\mathfrak l}}

\newcommand {\p}{{\mathfrak p}}
\newcommand {\q}{{\mathfrak q}}

\newcommand {\te}{{\mathfrak t}}
\newcommand {\ut}{{\mathfrak u}}

\newcommand {\gln}{{\mathfrak{gl}}_n}
\newcommand {\sln}{{\mathfrak{sl}}_n}
\newcommand {\sltn}{{\mathfrak{sl}}_{2n}}
\newcommand {\gltn}{{\mathfrak{gl}}_{2n}}

\newcommand {\spn}{{\mathfrak{sp}}_{2n}}
\newcommand {\sono}{{\mathfrak{so}}_{2n+1}}

\newcommand {\ind}{{\mathsf{ind\,}}}
\newcommand {\Lie}{{\mathrm{Lie\,}}}

\newcommand {\rk}{{\mathsf{rk\,}}}

\newcommand {\un}{\underline}
\newcommand {\GR}[2]{{\textrm{{\sf #1}}}_{#2}}
\newcommand {\beq}{\begin{equation}}
\newcommand {\eeq}{\end{equation}}

\newcommand {\bbk}{\Bbbk}

\newcommand{\gt}{\mathfrak}

\newcommand {\ca}{{\mathcal A}}

\newcommand {\BN}{{\mathbb N}}

\newcommand{\eus}{\EuScript}

\renewcommand{\le}{\leqslant}
\renewcommand{\ge}{\geqslant}

\newfam\eusfam%
\font\euszw=eusm10 scaled 1200%
\font\eusac=eusm7 scaled 1200%
\font\eusacc=eusm7 scaled 1000%
\textfont\eusfam=\euszw\scriptfont\eusfam=\eusac%
\scriptscriptfont\eusfam=\eusacc%
\begin{document}
\hfill {\scriptsize July 29, 2016}
\vskip1ex

\title{On seaweed subalgebras and meander graphs in type {\sf C}}
\author[D.\,Panyushev]{Dmitri I.~Panyushev}
\address[D.P.]{Institute for Information Transmission Problems of the Russian Academy of Sciences, 
Bolshoi Karetnyi per. 19,  Moscow 127051, Russia}
\email{panyushev@iitp.ru}
\author[O.\,Yakimova]%
{Oksana S.~Yakimova}
\address[O.Y.]{Institut f\"ur Mathematik, Friedrich-Schiller-Universit\"at Jena,  D-07737 Jena, 
Deutschland}
\email{oksana.yakimova@uni-jena.de}
\thanks{The first author gratefully acknowledges the hospitality of MPIM (Bonn) during the preparation of the article. The second author is partially supported by the DFG priority programme SPP 1388 
``Darstellungstheorie" and the Graduiertenkolleg GRK 1523 ``Quanten- und Gravitationsfelder".}

\maketitle

\begin{abstract}
In 2000, Dergachev and Kirillov introduced subalgebras of "seaweed type" in $\gt{gl}_n$(or $\gt{sl}_n$) and computed their index using certain graphs. In this article, those 
graphs are called type-{\sf A} meander graphs. Then the subalgebras of seaweed type, or just  
"seaweeds", have been defined by Panyushev (2001) for arbitrary simple Lie algebras.
Namely, if $\mathfrak p_1,\mathfrak p_2\subset\mathfrak g$ are parabolic subalgebras such that $\mathfrak p_1+\mathfrak p_2=\mathfrak g$, then $\mathfrak q=\mathfrak p_1\cap\mathfrak p_2$ is a seaweed in $\mathfrak g$. 
If $\mathfrak p_1$ and $\mathfrak p_2$ are ``adapted'' to a fixed triangular decomposition of $\mathfrak g$, then $\mathfrak q$ is said to be standard. The number of standard seaweeds is finite.
A general algebraic formula for the index of seaweeds has been proposed 
by Tauvel and Yu (2004) and then proved by Joseph (2006).

In this paper, elaborating on the ``graphical'' approach of Dergachev and Kirillov, we introduce the type-{\sf C} meander graphs, i.e., the graphs associated with the standard seaweed subalgebras of $\gt{sp}_{2n}$, and give a formula for the index in terms of these graphs. We also note that the very same graphs can be used in case of the odd orthogonal Lie algebras. 

Recall that $\mathfrak q$ is called Frobenius, if the index of $\mathfrak q$ equals 0. 
We  provide several applications of our formula to Frobenius seaweeds in $\gt{sp}_{2n}$. In particular, using a natural partition of the set 
$\boldsymbol{\eus {F}}_n$ of standard Frobenius seaweeds, we prove that $\#\boldsymbol{\eus {F}}_n$ strictly increases for the passage from $n$ to $n+1$. The similar monotonicity question is open for the standard Frobenius seaweeds in $\gt{sl}_n$, even for the passage from $n$ to $n+2$.
\end{abstract}

\section{Introduction}   
\label{sect:intro}

\noindent 
The index of an (algebraic) Lie algebra $\q$, $\ind\q$, is the minimal dimension of the stabilisers for the 
coadjoint representation of $\q$. It  can be regarded as a generalisation of the notion of rank. That is, 
$\ind\q$ equals the rank of $\q$, if $\q$ is reductive. In \cite{dk00}, the index of the subalgebras of 
"seaweed type" in $\gln$ (or $\sln$) has been computed using certain {graphs}. In this article, those 
graphs are called {\it type-{\sf A} meander  graphs}. Then the subalgebras of seaweed type, or just  
{\it seaweeds}, have been defined and studied for an  arbitrary simple Lie algebra $\g$~\cite{Dima01}. 
Namely, if $\p_1,\p_2\subset \g$ are parabolic subalgebras such that $\p_1+\p_2=\g$, then 
$\q=\p_1\cap\p_2$ is a seaweed in $\g$. If $\p_1$ and $\p_2$ are ``adapted'' to a fixed 
triangular decomposition of $\g$, then $\q$ is said to be standard, see Section~\ref{sect:generalities} for 
details. A general algebraic formula for the index of seaweeds has been proposed 
in~\cite[Conj.\,4.7]{ty-AIF} and then proved in~\cite[Section\,8]{jos}.

In this paper, elaborating on the ``graphical'' approach of \cite{dk00}, we introduce the {\it type-{\sf C} 
meander  graphs}, i.e., the graphs associated with the standard seaweed subalgebras of $\spn$, and give a formula for the index 
in terms of these graphs.  
Although the seaweeds in $\spn$ are our primary object in Sections~\ref{sect:generalities}--\ref{sect:applic} , we note that the very same graphs can be used in case of the odd orthogonal Lie algebras, see Section~\ref{sect:odd-orth}. 

Recall that $\q$ is called {\it Frobenius}, if $\ind\q=0$. Frobenius Lie algebras are very important in mathematics, because of their connection with the Yang-Baxter equation. We  provide some applications of our formula to Frobenius seaweeds in $\spn$. Let $\boldsymbol{\eus F}_n$ denote the set of standard Frobenius seaweeds of $\spn$.
For a natural partition $\boldsymbol{\eus F}_n=\bigsqcup_{k=1}^n \boldsymbol{\eus F}_{n,k}$ (see Section~\ref{sect:applic} for details), we construct the 
embeddings $\boldsymbol{\eus F}_{n,k}\hookrightarrow \boldsymbol{\eus F}_{n+1,k+1}$ for all $n,k\ge 1$.
Since $\boldsymbol{\eus F}_{n+1,1}$ does not meet the image of the induced embedding 
$\boldsymbol{\eus F}_n\hookrightarrow \boldsymbol{\eus F}_{n+1}$ and $\#(\boldsymbol{\eus F}_{n+1,1})>0$, this implies that $\#(\boldsymbol{\eus F}_{n})< \#(\boldsymbol{\eus F}_{n+1})$.
The similar monotonicity question is open for the standard Frobenius seaweeds in $\gt{sl}_n$, even for the passage from $n$ to $n+2$.
We also show that $\boldsymbol{\eus F}_{n,1}$ and $\boldsymbol{\eus F}_{n,2}$ are related to certain Frobenius seaweeds in $\sln$.

The ground field is algebraically closed and of characteristic zero. 

\section{Generalities on seaweed subalgebras and meander graphs}   
\label{sect:generalities}

\noindent 
Let $\p_1$ and $\p_2$ be two parabolic subalgebras of a simple Lie algebra $\g$.
If $\p_1 +\p_2=\g$, then $\p_1\cap\p_2$ is called a {\it seaweed subalgebra\/} or just {\it seaweed\/} in $\g$ (see \cite{Dima01}).
The set of seaweeds includes all parabolics (if $\p_2=\g$), all Levi subalgebras (if $\p_1$ and $\p_2$ are opposite), and many interesting non-reductive subalgebras.
We assume that $\g$ is equipped with a fixed triangular decomposition, so 
that there are two opposite Borel subalgebras $\be$ and $\be^-$, and a Cartan subalgebra 
$\te=\be\cap\be^-$. Without loss of generality,  we may also assume that  
$\p_1\supset \be$ (i.e., $\p_1$ is standard) and $\p_2=\p_2^-\supset \be^-$ (i.e., $\p_2$ is opposite-standard). Then the seaweed $\q=\p_1\cap\p_2^-$ is
said to be {\it standard}, too.
Either of these parabolics is determined by a subset of $\Pi$, the set of simple roots associated with $(\be,\te)$. 
Therefore, a standard seaweed  is determined by two arbitrary subsets of $\Pi$, see \cite[Sect.\,2]{Dima01} for details.

For classical Lie algebras $\sln$ and  $\spn$,  we exploit the usual numbering of $\Pi$, which allows
us to identify the standard and opposite-standard parabolic subalgebras with certain compositions related to 
$n$. It is also more convenient to deal with $\gln$ in place 
of $\sln$. 

{\bf I.} \ $\g=\gln$. We work with the obvious triangular decomposition of $\gln$, where $\be$ consists of the upper-triangular matrices.
If $\p_1\supset\be$ and the standard Levi subalgebra of $\p_1$ is 
$\gt{gl}_{a_1}{\oplus}\ldots\oplus\gt{gl}_{a_s}$, then we set $\p_1=\p(\un{a})$, where $\un{a}=(a_1,a_2,\dots,a_s)$. 
Note that $a_1+\dots +a_s=n$ and all $a_i \ge 1$.
Likewise, if $\p_2^-\supset\be^-$ is represented by 
a composition $\un{b}=(b_1,\dots,b_t)$ with $\sum b_j = n$, then the standard 
seaweed $\p_1\cap\p_2^-\subset \gln$ is  
denoted by $\q^{\sf A}(\un{a}{\mid} \un{b})$.  
The corresponding type-{\sf A} meander graph $\Gamma=\Gamma^{\sf A}(\un{a}{\mid}\un{b})$ 
is defined by the following rules:
\par\textbullet\quad
$\Gamma$ has $n$ consecutive vertices on a horizontal line numbered from $1$ up to $n$.
\par\textbullet\quad The parts of $\un{a}$ determine the set of pairwise disjoint arcs (edges) that are drawn \un{\sl above\/} the horizontal line.
Namely, part $a_1$ determines $[a_1/2]$ consecutively embedded arcs above the 
nodes $1,\dots,a_1$, where the widest arc joins vertices 1 and $a_1$, the following joins $2$ and $a_1-1$, etc. If $a_1$ is odd, then the middle vertex $(a_1+1)/2$ acquires no arc at all. 
Next,  part $a_2$ determines $[a_2/2]$ embedded
arcs above the nodes $a_1+1,\dots,a_1+a_2$, etc.
\par\textbullet\quad The arcs corresponding to $\un{b}$ are drawn following the same rules,
but \un{\sl below\/} the horizontal line. 

It follows that the degree of each vertex in $\Gamma$ is at most $2$ and each connected component of 
$\Gamma$ is homeomorphic to either a circle or a segment. (An isolated vertex is also a segment!) 
By~\cite{dk00}, the index of $\q^{\sf A}(\un{a}{\mid} \un{b})$ can be computed via
$\Gamma=\Gamma^{\sf A}(\un{a}{\mid}\un{b})$ as follows: 
\beq  \label{eq:index-A}
\ind\q^{\sf A}(\un{a}{\mid} \un{b})=2{\cdot}\text{(number of cycles in $\Gamma$)} + 
\text{(number of segments in $\Gamma$)} .
\eeq
\begin{rmk} \label{rem:ind-gl-sl}
Formula~\eqref{eq:index-A} gives the index of a seaweed in $\gln$, not in $\sln$.
However, 
if $\q\subset\gln$ is a seaweed, then $\q\cap\sln$ is a seaweed in $\sln$ and the respective mapping
$\q\mapsto \q\cap\sln$ is a bijection. Here
$\q=(\q\cap\sln)\oplus (\text{1-dim centre of $\gln$})$, hence
$\ind(\q\cap\sln)=\ind\q-1$. Since $\ind\q^{\sf A}(\un{a}{\mid} \un{b})\ge 1$ and the minimal value `1' is 
achieved if and only if $\Gamma$ is a sole segment, we also obtain a characterisation of the Frobenius seaweeds in $\sln$. 
\end{rmk}
\begin{ex}   \label{ex:old-A} 
$\Gamma^{\sf A}(5,2,2{\mid}  2,4,3)$=
\setlength{\unitlength}{0.021in}
\raisebox{-12\unitlength}{%
\begin{picture}(96,30)(-5,-11)
\multiput(0,3)(10,0){9}{\circle*{2}}
\put(20,5){\oval(40,20)[t]}
\put(20,5){\oval(20,10)[t]}
\put(55,5){\oval(10,5)[t]}
\put(75,5){\oval(10,5)[t]}
\put(5,1){\oval(10,5)[b]}
\put(35,1){\oval(30,15)[b]} 
\put(35,1){\oval(10,5)[b]}
\put(70,1){\oval(20,10)[b]}
\end{picture}
}  and the index of the corresponding seaweed in $\mathfrak{gl}_9$ (resp. $\mathfrak{sl}_9$) 
equals $3$ (resp. $2$).
\end{ex}

{\bf II.} \ $\g=\spn$. 
We use the embedding $\spn\subset\gltn$ such that
\[
\spn=\left\{\begin{pmatrix} \eus A & \eus B \\ \eus C & -\eus{\hat A} \end{pmatrix} \mid \eus A,\eus B,\eus C\in\gln,
\quad \eus B=\hat{\eus B}, \eus C=\hat{\eus C}\right\} ,
\] 
where $\eus A\mapsto\hat{\eus A}$ is the transpose with respect to the antidiagonal.
If $\tilde\be\subset\gltn$ (resp.~$\tilde\be^-$) is the set of upper- (resp. lower-) triangular matrices, then
$\be=\tilde\be\cap\spn$ and $\be^-=\tilde\be^-\cap\spn$ are our fixed Borel subalgebras of $\g=\spn$.
If $\p_1\supset\be$, then the standard Levi subalgebra of $\p$ is 
$\gt{gl}_{a_1}{\oplus}\ldots\oplus\gt{gl}_{a_s}\oplus\gt{sp}_{2d}$, where $a_1+\dots +a_s+d=n$,
all $a_i\ge 1$, and $d\ge 0$. Since $d$ is determined by $n$ and the `$\mathfrak{gl}$'  parts, 
$\p_1$ can be represented by $n$ and the composition $\un{a}=(a_1,\dots,a_s)$. We write
$\p_n(\un{a})$ for it. 
Likewise, if $\p_2^-$ is represented by another composition 
$\un{b}=(b_1,\dots,b_t)$ with $\sum b_j \le n$, then  $\p_1\cap\p_2^-$ is  denoted by
$\q_n^{\sf C}(\un{a}{\mid} \un{b})$.  
To a standard parabolic $\p_1=\p_n(\un{a})\subset\spn$, one can associate the parabolic subalgebra
$\tilde\p_1\subset\gltn$ that is represented by the symmetric composition 
$\un{\tilde a}=(a_1,\dots,a_s, 2d, a_s,\dots,a_1)$ of $2n$. In the matrix form, the standard Levi 
subalgebra of $\tilde\p_1$ has the consecutive diagonal blocks $\gt{gl}_{a_1},\ldots,\gt{gl}_{a_s},\gt{gl}_{2d},
\gt{gl}_{a_s},\ldots,\gt{gl}_{a_1}$ and, for the above embedding $\spn\subset\gltn$ and compatible triangular decompositions, 
one has $\p_1=\tilde\p_1\cap\spn$ (and likewise for $\p_2^-\subset\spn$ and $\tilde\p_2^-\subset\gltn$), see \cite[Sect.\,5]{Dima01} for details. If $\un{\tilde a}$ and $\un{\tilde b}$ are symmetric compositions of $2n$,  then the seaweed
$\q^{\sf A}(\un{\tilde a}\!\mid\!\un{\tilde b})\subset\gltn$ is said to be {\it symmetric}, too. The above construction 
provides a bijection between the standard seaweeds in $\spn$ and the symmetric standard seaweeds in 
$\gltn$ (or $\sltn)$.

We define the {\it type-{\sf C} meander graph\/} $\Gamma^{\sf C}_n(\un{a}\!\mid\!\un{b})$ for 
$\q^{\sf C}_n(\un{a}\!\mid\!\un{b})$ to be the type-{\sf A} meander graph of the corresponding symmetric 
seaweed $\tilde\q=\tilde\p_1\cap\tilde\p_2^-\subset \gltn$. Formally,
\[
   \Gamma^{\sf C}_n(\un{a}\!\mid\!\un{b})=\Gamma^{\sf A}(\un{\tilde a}\!\mid\!\un{\tilde b}).
\]
We indicate below new features of these graphs.
\par\textbullet\quad
$\Gamma^{\sf C}_n(\un{a}\!\mid\!\un{b})$ has $2n$ consecutive vertices on a horizontal line  numbered from $1$ up to $2n$.
\par\textbullet\quad Part $a_1$ determines $[a_1/2]$  embedded arcs above the 
nodes $1,\dots,a_1$. By symmetry, the same set of arcs appears above the vertices 
$2n-a_1+1,\dots, 2n$.
Next,  part $a_2$ determines $[a_2/2]$ embedded
arcs above the nodes $a_1+1,\dots,a_1+a_2$ and also the symmetric set of arcs above the nodes 
$2n-a_1-a_2+1,\dots 2n-a_1$, etc.
\par\textbullet\quad If $d=n-\sum a_i >0$, then there are $2d$ unused vertices in the middle, and we draw $d$ embedded arcs above them. This corresponds to part $2d$ that occurs in the middle of
$\un{\tilde a}$. 
The arcs corresponding to $\un{b}$ are depicted by the same rules, but
 \un{\sl below\/} the horizontal line. 
\par\textbullet\quad
A type-{\sf C} meander graph is symmetric with respect to the vertical line between the $n$-th 
and $(n+1)$-th vertices, and the symmetry w.r.t this line is denoted by $\sigma$. We also say that this
line is the $\sigma$-{\it mirror}. The arcs crossing the $\sigma$-mirror are said to be {\it central}. These are exactly the arcs corresponding to $d=n-\sum a_i$ and $d'=n-\sum b_j$.

Our main result is the following formula for the index in terms of the connected components  of 
$\Gamma^{\sf C}_n(\un{a}{\mid}\un{b})$: 
\beq  \label{eq:index-C}
\ind\q^{\sf C}_n(\un{a}{\mid} \un{b})=\text{(number of cycles)} + \frac{1}{2}\text{(number of segments that are not $\sigma$-stable)} .
\eeq

\vskip.8ex\noindent
To illustrate this formula, we recall that, for the parabolic subalgebra $\p$ with Levi part
$\gt{gl}_{a_1}{\oplus}\ldots{\oplus}\gt{gl}_{a_s}{\oplus}\gt{sp}_{2d}$, we have 
$\ind\p=\left[ \frac{a_1}{2} \right ]+\ldots+\left[ \frac{a_s}{2} \right ]+d$, see \cite[Theorem\,5.5]{Dima01}.
Here  $\p_2^-=\spn$ and 
the  composition $\un{b}$ is empty.
On the other hand, the graph $\Gamma^{\sf C}_n(\un{a}\!\mid\!\varnothing)$ has $n$ central arcs 
below the horizontal line corresponding to $\un{b}=\varnothing$. Hence
each part $a_i$ gives 
rise to $\left[ \frac{a_i}{2} \right ]$ cycles and, if $a_i$ is odd, to one 
additional segment, which is $\sigma$-invariant. 
The middle part corresponding to $\gt{sp}_{2d}$ gives rise to $d$ cycles.  
This clearly yields the same answer, cf. Example~\ref{ex:parab-graf}. Hence we already know that 
Eq.~\eqref{eq:index-C} is correct, if
$\q$ is a parabolic subalgebra, i.e., if $\un{a}=\varnothing$ or $\un{b}=\varnothing$. Note also that $\ind\p=0$ if and only if $d=0$ and all $a_i=1$, i.e., if $\p=\be$.

\begin{ex}   \label{ex:parab-graf}
Here $\un{a}=(2,3)$ and $n=7$ (hence $d=2$), and the $\sigma$-mirror is represented by the vertical dotted line.
It is easily seen that the only segment here is $\sigma$-stable and the total number of circles is $4$. (The circles are depicted by {\color{blue}blue} arcs). Hence
$\ind\p=4$.

\begin{figure}[htb]
\setlength{\unitlength}{0.025in}
\begin{center}
\begin{picture}(140,40)(-5,-27)

\put(-50,-5){$\Gamma^{\sf C}_7(2,3\!\mid\!\varnothing)$\ :}
\multiput(0,3)(10,0){14}{\circle*{2}}

\put(65,1){\oval(70,35)[b]}

{\color{blue}
\put(5,5){\oval(10,5)[t]}
\put(30,5){\oval(20,10)[t]}
\put(100,5){\oval(20,10)[t]}
\put(125,5){\oval(10,5)[t]}
\put(65,5){\oval(30,15)[t]}
\put(65,5){\oval(10,5)[t]}
\put(65,1){\oval(10,5)[b]}
\put(65,1){\oval(30,15)[b]}
\put(65,1){\oval(50,25)[b]}
\put(65,1){\oval(90,45)[b]}
\put(65,1){\oval(110,55)[b]}
\put(65,1){\oval(130,65)[b]}

}
\qbezier[40](65,-35),(65,-10),(65,15)
\end{picture}
\end{center}
\caption{The meander graph for a parabolic subalgebra of  $\mathfrak{sp}_{14}$}   \label{pikcha_C7}
\end{figure}
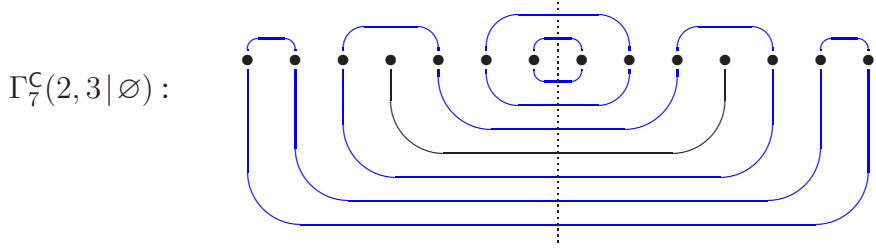
\end{ex}

\begin{rmk}
1) For both $\gln$ and $\spn$, one has $\q^\ast(\un{a}{\mid}\un{b})\simeq \q^\ast(\un{b}{\mid}\un{a})$. 
Hence one can freely choose what composition is going to appear first.

2) It is also true that $\q^\ast(\un{a}{\mid}\un{b})$ is reductive (i.e., a Levi subalgebra) if and only if
$\un{a}=\un{b}$.
\end{rmk}

{\bf Convention}. If $\q$ is a seaweed in either $\spn$ or $\gltn$, and the corresponding compositions are not specified, then the respective meander graph is denoted by $\Gamma^{\sf C}(\q)$ or $\Gamma^{\sf A}(\q)$.

\begin{rmk}                \label{mrs}
Let $\gt q$ be a seaweed in $\gt{sp}_{2n}$ or $\gt{gl}_n$. Then there is a point $\gamma\in\gt q^*$ such that 
the stabiliser $\gt q_\gamma\subset \gt q$ is a reductive subalgebra, see \cite{Dima03-b}. A Lie algebra 
possessing such a point in the dual space is said to be  {\it (strongly) quasi-reductive}~\cite{DKT}, see also 
\cite[Def.~2.1]{MY12}. One of the main results of 
\cite{DKT} states that if a Lie  algebra $\gt q=\Lie Q$ is strongly quasi-reductive, then there is a reductive stabiliser 
$Q_\gamma$ (with $\gamma\in \gt q^*$) such that any other reductive 
stabiliser $Q_\beta$  (with $\beta\in\gt q^*$) is contained in $Q_\gamma$ up to conjugation. In 
\cite{MY12} this subgroup $Q_\gamma$ is a called a {\it maximal reductive stabiliser}, MRS for short.
For a seaweed  $\gt q=\q^{\sf A}(\un{a}\!\mid\!\un{b})$, an MRS of $\gt q$ can be described in terms of 
$\Gamma^{\sf A}(\un{a}\!\mid\!\un{b})$ \cite[Theorem~5.3]{MY12}. A similar description is possible in type ${\sf C}$ if 
we use $\Gamma^{\sf C}_n(\un{a}\!\mid\!\un{b})$. It will appear elsewhere. 
\end{rmk}

\section{Symplectic meander graphs and the index of seaweed subalgebras}
\label{sect:main}

\noindent
In this section, we prove formula~\eqref{eq:index-C} on the index of the seaweed subalgebras of type {\sf C}.

Let us recall the inductive procedure for computing the index of seaweeds in a symplectic Lie 
algebra introduced by the first author~\cite{Dima01}. Suppose that $\un{a}=(a_1,\dots,a_s)$ and 
$\un{b}=(b_1,\dots,b_t)$ are two compositions with $\sum a_i \le n$ and $\sum b_j \le n$. Then we
consider the standard seaweed $\q^{\sf C}_n(\un{a}{\mid} \un{b})
\subset \spn$.

{\it\bfseries Inductive procedure}:

{\sl \bfseries 1.} If either $\un{a}$ or $\un{b}$ is empty, then $\q^{\sf C}_n(\un{a}{\mid} \un{b})$ is a parabolic subalgebra and the index is computed using \cite[Theorem\,5.5]{Dima01} (cf. also Introduction).

{\sl\bfseries 2.}  Suppose that both $\un{a}$ and $\un{b}$ are non-empty.
Without loss of generality, we can assume that $a_1\le b_1$. By \cite[Theorem\,5.2]{Dima01}, $\ind \q^{\sf C}_n(\un{a}{\mid} \un{b})$ can inductively be computed as follows:
\par
{\sf (i)}\quad If $a_1=b_1$, then $\q^{\sf C}_n(\un{a}{\mid} \un{b})\simeq \mathfrak{gl}_{a_1}\oplus \q^{\sf C}_{n-a_1}(a_2,\dots,a_s{\mid} b_2,\dots,b_t)$, hence
\par
\centerline{$
  \ind \q^{\sf C}_n(\un{a}{\mid} \un{b})=a_1+\ind \q^{\sf C}_{n-a_1}(a_2,\dots,a_s{\mid} b_2,\dots,b_t)$.}
\par
{\sf (ii)}\quad If $a_1<b_1$, then 
\[
\ind\q^{\sf C}_n(\un{a}{\mid}\un{b})=
\left\{\begin{array}{ll} 
\ind\q^{\sf C}_{n-a_1}(a_2,\dots,a_s{\mid} b_1-2a_1,a_1,b_2,\dots,b_t) 
& \ {\mathrm{if\ }} a_1\le b_1/2 ; \\
\ind\q^{\sf C}_{n-b_1+a_1}(2a_1-b_1,a_2,\dots,a_s{\mid} a_1,b_2,\dots,b_t) 
& \ {\mathrm{if\ }} a_1 > b_1/2  .
\end{array}\right.
\]
\par{\sf (iii)}\quad Step~{\sl\bfseries 2} terminates when one of the compositions becomes empty, i.e., one obtains a parabolic subalgebra in a smaller symplectic Lie algebra, where Step~{\sl\bfseries 1} applies.

\begin{rmk}   \label{rmk:another-ind}
Iterating transformations of the form {\sl\bfseries 2}{\sf (ii)} yields a formula that does not require considering cases, see~\cite[Theorem\,5.3]{Dima01}. Namely, if $a_1<b_1$, then
$\ind \q^{\sf C}_n(\un{a}{\mid} \un{b})=
\ind \q^{\sf C}_{n-a_1}(\un{a}'{\mid} \un{b}')$, where $\un{a}'=(a_2,\dots,a_s)$, 
$\un{b}'=(b_1',b_1'',b_2,\dots,b_t)$, and $b_1'$ and $b_1''$ are defined as follows.
Let $p$ be the unique integer such that
$\displaystyle \frac{p}{p+1} < \frac{a_1}{b_1}\le \frac{p+1}{p+2}$. Then 
$b_1'=(p+1)b_1-(p+2)a_1\ge 0$ and $b_2'=(p+1)a_1-pb_1>0$. 
(If $b_1'=0$, then it has to be omitted.)
\end{rmk}
\begin{thm}   
\label{thm:main}
Let\/ $\q=\q^{\sf C}_n(\un{a}{\mid} \un{b})$ be a seaweed in $\spn$ and 
$\Gamma^{\sf C}(\q)=\Gamma^{\sf C}_n(\un{a}{\mid}\un{b})$  the type-{\sf C} meander graph associated 
with $\q$. Then 
\[
\ind\q^{\sf C}_n(\un{a}{\mid} \un{b})=\#\text{\{cycles of\/ $\Gamma^{\sf C}_n(\un{a}{\mid}\un{b})$\}} + 
\frac{1}{2}\#\text{\{segments of\/ $\Gamma^{\sf C}_n(\un{a}{\mid}\un{b})$ that are not $\sigma$-stable\}} .
\]
\end{thm}
\begin{proof}
Our argument exploits the above {\it inductive procedure}. 
Let us temporarily write $\mathcal{T}_n(\un{a}{\mid}\un{b})$ for the topological quantity in the right hand side 
of the formula.
Let us prove that for the pairs of seaweeds occurring in either {\sl\bfseries 2}{\sf (i)} or 
{\sl\bfseries 2}{\sf (ii)} of the inductive procedure, the required topological quantity behaves accordingly.    

If $a_1=b_1$ and $\gt{gl}_{a_1}$ is a direct summand of $\gt q$, then the index of
$\q^{\sf C}_{n-a_1}(a_2,\dots,a_s{\mid} b_2,\dots,b_t)$ decreases by $a_1$; on the other hand, 
$\Gamma^{\sf C}_{n-a_1}(a_2,\dots,a_s{\mid} b_2,\dots,b_t)$  is obtained from $\Gamma^{\sf C}(\q)$ by
deleting $2\left[ \frac{a_1}{2} \right ]$ cycles (and two segments, which are not 
$\sigma$-invariant in case $a_1$ is odd). This is in perfect agreement with the formula.

If $a_1<b_1$, then one step of $\mathfrak{sp}$-reduction for $\q$ is equivalent to two steps of 
$\mathfrak{gl}$-reduction for the meander graph of $\Gamma^{\sf A}(\tilde{\q})$, where 
$\tilde{\gt q}$ is the corresponding symmetric seaweed in $\gt{gl}_{2n}$. These two ``symmetric'' steps are 
applied one after another 
to the left and right hand sides of $\Gamma^{\sf A}(\tilde{\q})=\Gamma^{\sf C}(\gt q)$. According 
to~\cite[Lemma\,5.4(i)]{MY12}, the $\mathfrak{gl}$-reduction does not change the topological structure of 
the graph. Hence 
$\mathcal{T}_n(\un{a}{\mid}\un{b})=\mathcal{T}_{n-a_1}(\un{a}'{\mid}\un{b}')$.

Since we have already observed (in Section~\ref{sect:generalities}) that our formula holds for the parabolic subalgebras, the result follows.
\end{proof}
\begin{ex}   \label{ex:reductio}
For the seaweed $\q_{10}(3,3{\mid} 4,5)$ in $\mathfrak{sp}_{20}$, the recursive formula of Remark~\ref{rmk:another-ind} 
yields the following chain  of reductions:
\[
  \q=\q^{\sf C}_{10}(3,3{\mid} 4,5)\leadsto \q^{\sf C}_{7}(3{\mid} 1,5)\leadsto \q^{\sf C}_{6}(1,1{\mid} 5)\leadsto 
  \q^{\sf C}_{5}(1{\mid} 3,1)\leadsto \q^{\sf C}_{4}(\varnothing{\mid} 1,1,1) .
\]
The last term represents the minimal parabolic subalgebra of $\mathfrak{sp}_8$ corresponding to the unique long simple root. The respective graphs are gathered in Figure~\ref{pikcha_C10}. It is readily
seen that both ends of the graphs undergo the symmetric transformations on each step; also all the segments are $\sigma$-stable and the total number of cycles equals 1. Thus, $\ind\q=1$.
 
\begin{figure}[htb]
\setlength{\unitlength}{0.025in}
\begin{center}

\begin{picture}(200,35)(-20,-10)
\put(-45,0){$\Gamma^{\sf C}_{10}(3,3\!\mid\! 4,5)$\ :}
\multiput(0,3)(10,0){20}{\circle*{2}}

\put(10,5){\oval(20,10)[t]}
\put(40,5){\oval(20,10)[t]}
{\color{blue}\put(95,5){\oval(10,5)[t]}}
\put(95,5){\oval(30,15)[t]}
\put(95,5){\oval(50,25)[t]}
\put(95,5){\oval(70,35)[t]}
\put(150,5){\oval(20,10)[t]}
\put(180,5){\oval(20,10)[t]}

\put(15,1){\oval(30,15)[b]}
\put(15,1){\oval(10,5)[b]}
\put(60,1){\oval(20,10)[b]}
\put(60,1){\oval(40,20)[b]}
{\color{blue}\put(95,1){\oval(10,5)[b]}}
\put(130,1){\oval(20,10)[b]}
\put(130,1){\oval(40,20)[b]}
\put(175,1){\oval(10,5)[b]}
\put(175,1){\oval(30,15)[b]}

\qbezier[45](95,-15),(95,5),(95,25)
\end{picture}

\begin{picture}(200,45)(-20,-10)
\put(-40,0){$\Gamma^{\sf C}_{7}(3\!\mid\! 1,5)$\ :}
\multiput(30,3)(10,0){14}{\circle*{2}}

\put(40,5){\oval(20,10)[t]}
{\color{blue}\put(95,5){\oval(10,5)[t]}}
\put(95,5){\oval(30,15)[t]}
\put(95,5){\oval(50,25)[t]}
\put(95,5){\oval(70,35)[t]}
\put(150,5){\oval(20,10)[t]}

\put(60,1){\oval(20,10)[b]}
\put(60,1){\oval(40,20)[b]}
{\color{blue}\put(95,1){\oval(10,5)[b]}}
\put(130,1){\oval(20,10)[b]}
\put(130,1){\oval(40,20)[b]}

\qbezier[45](95,-15),(95,5),(95,25)
\end{picture}

\begin{picture}(200,45)(-20,-10)
\put(-40,0){$\Gamma^{\sf C}_{6}(1,1\!\mid\! 5)$\ :}
\multiput(40,3)(10,0){12}{\circle*{2}}

{\color{blue}\put(95,5){\oval(10,5)[t]}}
\put(95,5){\oval(30,15)[t]}
\put(95,5){\oval(50,25)[t]}
\put(95,5){\oval(70,35)[t]}

\put(60,1){\oval(20,10)[b]}
\put(60,1){\oval(40,20)[b]}
{\color{blue}\put(95,1){\oval(10,5)[b]}}
\put(130,1){\oval(20,10)[b]}
\put(130,1){\oval(40,20)[b]}

\qbezier[45](95,-15),(95,5),(95,25)
\end{picture}

\begin{picture}(200,45)(-20,-10)
\put(-40,0){$\Gamma^{\sf C}_{5}(1\!\mid\! 3,1)$\ :}
\multiput(50,3)(10,0){10}{\circle*{2}}

{\color{blue}\put(95,5){\oval(10,5)[t]}}
\put(95,5){\oval(30,15)[t]}
\put(95,5){\oval(50,25)[t]}
\put(95,5){\oval(70,35)[t]}

\put(60,1){\oval(20,10)[b]}
{\color{blue}\put(95,1){\oval(10,5)[b]}}
\put(130,1){\oval(20,10)[b]}

\qbezier[45](95,-15),(95,5),(95,25)
\end{picture}

\begin{picture}(200,45)(-20,-10)
\put(-40,0){$\Gamma^{\sf C}_{4}(\varnothing\!\mid\! 1,1,1)$\ :}
\multiput(60,3)(10,0){8}{\circle*{2}}

{\color{blue}\put(95,5){\oval(10,5)[t]}}
\put(95,5){\oval(30,15)[t]}
\put(95,5){\oval(50,25)[t]}
\put(95,5){\oval(70,35)[t]}

{\color{blue}\put(95,1){\oval(10,5)[b]}}
 
 \qbezier[45](95,-15),(95,5),(95,25)
\end{picture}
\end{center}
\caption{The reduction steps for a seaweed subalgebra of  $\mathfrak{sp}_{20}$}   \label{pikcha_C10}
\end{figure}
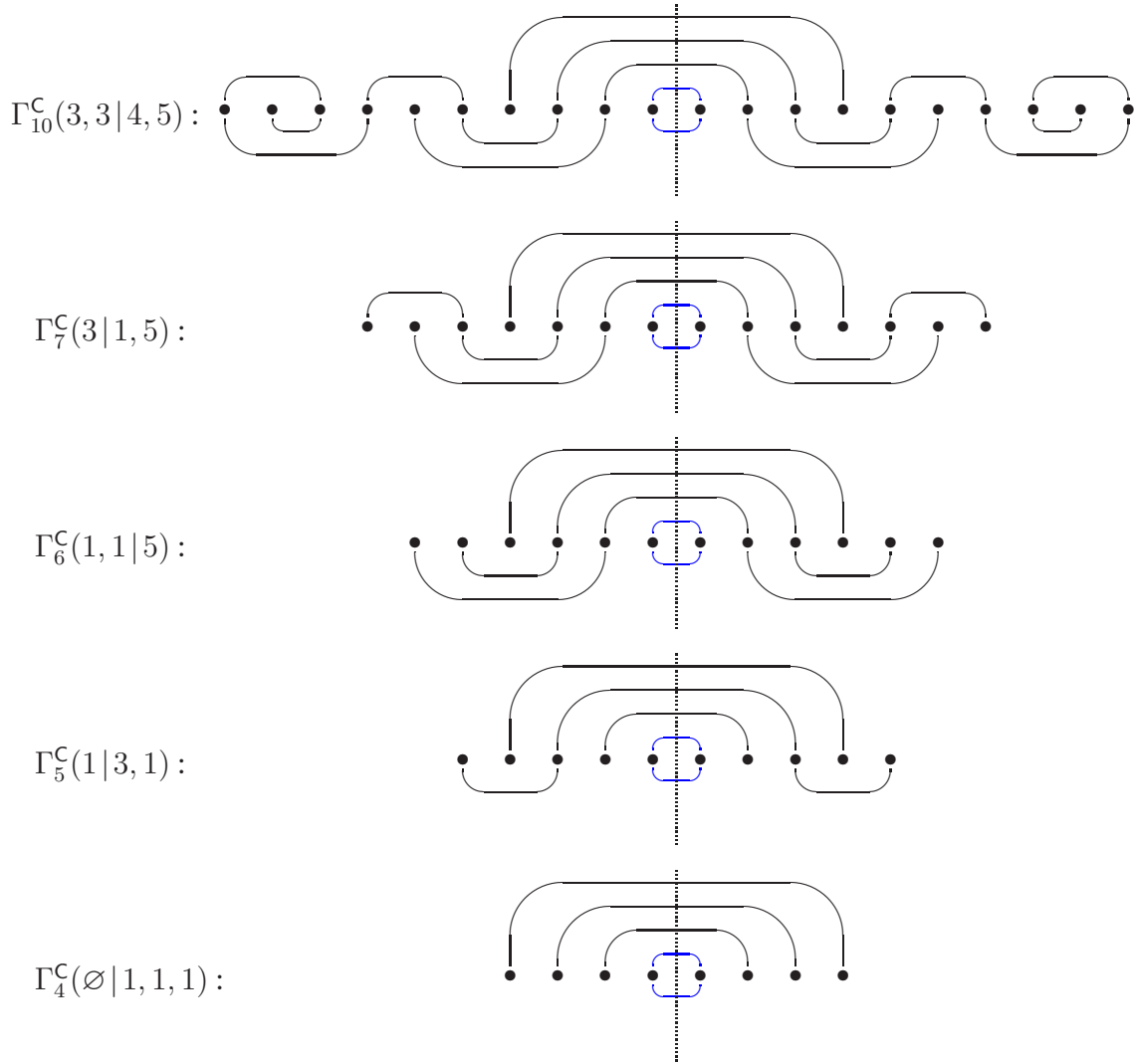

One can notice that each reduction step consists of contracting certain arcs starting from some end vertices of
a meander graph. Clearly, such a procedure does not change the topological structure of the graph, and this is 
exactly how Lemma~5.4(i) in \cite{MY12} has been proved.
\end{ex}

\begin{ex}   \label{ex:frob-16}
In Figure~\ref{pikcha_C8}, one finds the graph of a  seaweed in $\mathfrak{sp}_{16}$ of index $1$. The segments that are not 
$\sigma$-stable are depicted by {\color{red}red} arcs.
\begin{figure}[htb]
\setlength{\unitlength}{0.025in}
\begin{center}
\begin{picture}(160,25)(-5,-10)
\multiput(0,3)(10,0){16}{\circle*{2}}
\put(-45,0){$\Gamma^{\sf C}_{8}(3,4\!\mid\! 5,3)$\ :}

\put(10,5){\oval(20,10)[t]}
\put(45,5){\oval(10,5)[t]}
{\color{red}\put(45,5){\oval(30,15)[t]}}
\put(75,5){\oval(10,5)[t]}
\put(105,5){\oval(10,5)[t]}
{\color{red}\put(105,5){\oval(30,15)[t]}}
\put(140,5){\oval(20,10)[t]}

{\color{red}\put(20,1){\oval(20,10)[b]}}
\put(20,1){\oval(40,20)[b]}
\put(60,1){\oval(20,10)[b]}
\put(90,1){\oval(20,10)[b]}
{\color{red}\put(130,1){\oval(20,10)[b]}}
\put(130,1){\oval(40,20)[b]}

\qbezier[25](75,-11),(75,2),(75,15)
\end{picture}
\end{center}
\caption{A  seaweed subalgebra of  $\mathfrak{sp}_{16}$ with index $1$}   \label{pikcha_C8}
\end{figure}
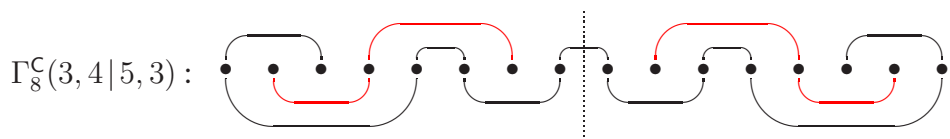
\end{ex}

\section{Applications of symplectic meander graphs}
\label{sect:applic}

\noindent
In this section, we present some applications of Theorem~\ref{thm:main}. We begin with a simple
property of the index.

\begin{lm}  \label{lm:simple-prop}
If \ $\sum a_i<n$ and \ $\sum b_j< n$, then 
$
\ind\q^{\sf C}_n(\un{a}{\mid}\un{b})=(n-n')+ \ind\q^{\sf C}_{n'}(\un{a}{\mid}\un{b}) ,
$
where $n'=\max\{\sum a_i, \sum b_j\}$.
\end{lm}
\begin{proof}
Here $\Gamma^{\sf C}_n(\un{a}{\mid}\un{b})$ contains  $n-n'$ arcs crossing the 
$\sigma$-mirror on the {\bf both} sides of 
the horizontal line. They form $n-n'$ central circles, and removing these circles reduces the index by 
$n-n'$ and yields the graph $\Gamma^{\sf C}_{n'}(\un{a}{\mid}\un{b})$.
\end{proof}

\noindent
Recall that a Lie algebra $\q$ is {\it Frobenius}, if $\ind \q=0$. In the rest of the section, we apply
Theorem~\ref{thm:main} to studying Frobenius seaweeds. 
Clearly, if $\q^{\sf C}_n(\un{a}{\mid}\un{b})$ is Frobenius, then 
$\Gamma^{\sf C}_n(\un{a}{\mid}\un{b})$ has only  $\sigma$-stable  segments and no cycles. 
Another consequence of Theorem~\ref{thm:main} is the following necessary condition.
\begin{lm}   \label{lm:simple-cons}
If\/  $\q^{\sf C}_n(\un{a}{\mid}\un{b})$ is Frobenius, then either $\sum a_i<n$ and $\sum b_j=n$ or 
vice versa.
\end{lm}
\begin{proof}
If $\sum a_i<n$ and $\sum b_j<n$, then the index is positive in view of Lemma~\ref{lm:simple-prop}.
If $\sum a_i=\sum b_j=n$, then
there are no arcs crossing the $\sigma$-mirror. Therefore $\Gamma^{\sf C}_n(\un{a}{\mid}\un{b})$ 
consists of two disjoint $\sigma$-symmetric parts, and the topological quantity of 
Theorem~\ref{thm:main} cannot be equal to $0$.
(More precisely, in the second case $\q^{\sf C}_n(\un{a}{\mid}\un{b})$ is isomorphic to the seaweed 
$\q^{\sf A}(\un{a}{\mid}\un{b})$ in $\gln$, and $\ind\q\ge 1$ for all seaweeds $\q\subset\gln$, see 
Remark~\ref{rem:ind-gl-sl}.)
\end{proof}
Graphically, Lemma~\ref{lm:simple-cons} means that, for a  Frobenius seaweed, 
one must have some central arcs (=\,arcs crossing the $\sigma$-mirror) on one side of 
the horizontal line in the meander graph, and then there has to be no central arcs on the other side.
The number of central arcs can vary from $1$ to $n$ (the last possibility represents the case in which 
one of the parabolics is the Borel subalgebra). 
Let $\boldsymbol{\eus {F}}_{n,k}$ denote the set of standard Frobenius seaweeds whose meander 
graph contains $k$ central arcs.  
Then $\boldsymbol{\eus {F}}_{n}=\bigsqcup_{k=1}^n \boldsymbol{\eus {F}}_{n,k}$ is the set of all standard 
Frobenius seaweeds in $\spn$.  If $\q^{\sf C}_n(\un{a}{\mid}\un{b})$ lies in 
$\boldsymbol{\eus {F}}_{n,k}$, then
so is $\q^{\sf C}_n(\un{b}{\mid}\un{a})$. As we are interested in essentially different meander graphs, 
we will not distinguish graphs and algebras corresponding to $(\un{a}{\mid}\un{b})$ and 
$(\un{b}{\mid}\un{a})$.
Set $\boldsymbol{F}_{n,k}=\#(\boldsymbol{\eus {F}}_{n,k}/{\sim})$ and 
$\boldsymbol{F}_{n}=\#(\boldsymbol{\eus {F}}_{n}/\sim)$, where $\sim$ is the corresponding 
equivalence relation.  Then \\
$\boldsymbol{F}_{n,n}=1$, \ 
$\boldsymbol{F}_{n,n{-}1}=\begin{cases}  1, &  n=2; \\   2, & n\ge 3. \end{cases}$,  and  \ 
$\boldsymbol{F}_{n,n{-}2}=\begin{cases}  2, &  n=3; \\   4, & n=4; \\
5, & n\ge 5. \end{cases}$. 
\\ \indent
It  follows from Lemma~\ref{lm:simple-cons} that if 
$\q^{\sf C}_n(\un{a}{\mid}\un{b})\in \boldsymbol{\eus {F}}_{n}$ and $\sum b_j=n$, then the integer
$k$ such that $\q^{\sf C}_n(\un{a}{\mid}\un{b})\in \boldsymbol{\eus {F}}_{n,k}$ is determined as 
$k=n-\sum a_i$.
In Figure~\ref{pikcha_C7-2}, one finds the meander graphs of Frobenius seaweeds in $\mathfrak{sp}_{14}$ with $k=1$ and $2$.

\begin{figure}[htb]
\setlength{\unitlength}{0.025in}
\begin{center}
\begin{picture}(140,25)(-10,-10)
\multiput(0,3)(10,0){14}{\circle*{2}}
\put(-45,0){$\Gamma^{\sf C}_{7}(2,4\!\mid\! 4,3)$\ :}

\put(5,5){\oval(10,5)[t]}
\put(35,5){\oval(30,15)[t]}
\put(35,5){\oval(10,5)[t]}
\put(65,5){\oval(10,5)[t]}
\put(95,5){\oval(30,15)[t]}
\put(95,5){\oval(10,5)[t]}
\put(125,5){\oval(10,5)[t]}

\put(15,1){\oval(30,15)[b]}
\put(15,1){\oval(10,5)[b]}
\put(50,1){\oval(20,10)[b]}
\put(80,1){\oval(20,10)[b]}
\put(115,1){\oval(10,5)[b]}
\put(115,1){\oval(30,15)[b]}

\qbezier[25](65,-11),(65,2),(65,15)
\end{picture}

\begin{picture}(140,28)(-10,-5)
\multiput(0,3)(10,0){14}{\circle*{2}}
\put(-45,0){$\Gamma^{\sf C}_{7}(3,2\!\mid\! 2,5)$\ :}

\put(10,5){\oval(20,10)[t]}
\put(35,5){\oval(10,5)[t]}
\put(65,5){\oval(10,5)[t]}
\put(65,5){\oval(30,15)[t]}
\put(95,5){\oval(10,5)[t]}
\put(120,5){\oval(20,10)[t]}

\put(5,1){\oval(10,5)[b]}
\put(40,1){\oval(20,10)[b]}
\put(40,1){\oval(40,20)[b]}
\put(90,1){\oval(40,20)[b]}
\put(90,1){\oval(20,10)[b]}
\put(125,1){\oval(10,5)[b]}

\qbezier[25](65,-11),(65,2),(65,15)
\end{picture}

\end{center}
\caption{Frobenius seaweed subalgebras of  $\mathfrak{sp}_{14}$}   
\label{pikcha_C7-2}
\end{figure}
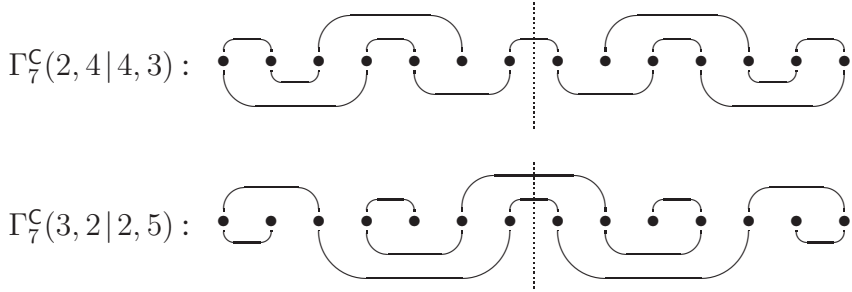

\begin{lm}   \label{lm:k-connected}
If $\q \in \boldsymbol{\eus {F}}_{n,k}$, then $\Gamma^{\sf C}(\q)$ has exactly $k$ connected 
components ($\sigma$-stable segments) corresponding to the central arcs. Furthermore, the total number of arcs in $\Gamma^{\sf C}(\q)$ equals $2n-k$.
\end{lm}
\begin{proof} 
1) \ Let $\ca_i$ be the $i$-th central arc and $\Gamma_i$ the connected component of 
$\Gamma^{\sf C}(\q)$ that contains $\ca_i$. Each $\Gamma_i$ is a $\sigma$-stable segment.
\par\textbullet\quad If $\Gamma_i=\Gamma_j$ for $i\ne j$, then continuations of $\ca_i$ and $\ca_j$ meet 
somewhere in the left hand half of $\Gamma^{\sf C}(\q)$. By symmetry, the same happens in the right hand 
half, which produces a cycle. Hence the connected components $\Gamma_1,\dots,\Gamma_k$ must be 
different. 
\par\textbullet\quad Assume that there exists yet another connected component $\Gamma_{k+1}$. Then
it belongs to only one half of $\Gamma^{\sf C}(\q)$. By symmetry, there is also the ``same'' component
$\Gamma_{k+2}$ in the other half of $\Gamma^{\sf C}(\q)$. This would imply that $\ind\q>0$.

2) \ Since the graph $\Gamma^{\sf C}(\q)$ has $2n$ vertices and is a disjoint union of $k$ trees, the number of edges (arcs) must be $2n-k$.
\end{proof}

\begin{lm}   \label{lm:monotonicity}
For any $k\ge1$, there is an injective map
$\boldsymbol{\eus F}_{n,k}\to  \boldsymbol{\eus F}_{n+1,k+1}$. Moreover, 
$\boldsymbol{F}_{n+1}> \boldsymbol{F}_{n}$, that is,
the total number of Frobenius seaweeds strictly increases under the passage from $n$ to $n+1$.
\end{lm}
\begin{proof}
For any $\q\in \boldsymbol{\eus F}_{n,k}$ ($k\ge 1$), we can add two new vertices in the middle of
$\Gamma^{\sf C}(\q)$ and connect them by an arc (on the appropriate side!). This yields an injective
mapping $\boldsymbol{\eus F}_{n,k}\to \boldsymbol{\eus F}_{n+1,k+1}$ for any $k\ge 1$ and thereby an injection $i_n: \boldsymbol{\eus F}_{n}\hookrightarrow \boldsymbol{\eus F}_{n+1}$.

Since $\boldsymbol{\eus F}_{n+1,1}$ does not intersect the image of $i_n$, the second assertion follows from the fact that $\boldsymbol{F}_{n+1,1}>0$ for any $n\ge 0$, see example below.
\end{proof}

{\it \bfseries Example.} We point out an explicit element 
$\q^{\sf C}_n(\un{a}{\mid}\un{b})\in\boldsymbol{\eus F}_{n,1}$. For $n=2k$, 
one takes $\un{a}=(2^{k})$  and $\un{b}=(1,2^{k-1})$. For $n=2k+1$, one takes
$\un{a}=(2^{k})$  and $\un{b}=(1,2^{k})$. For $n=4$, the meander graph is:
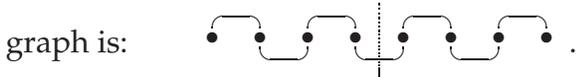
\begin{picture}(160, 20)(-10,0)
\setlength{\unitlength}{0.025in}
\multiput(10,3)(10,0){8}{\circle*{2}}

\put(15,5){\oval(10,5)[t]}
\put(35,5){\oval(10,5)[t]}
\put(55,5){\oval(10,5)[t]}
\put(75,5){\oval(10,5)[t]}

\put(25,1){\oval(10,5)[b]}
\put(45,1){\oval(10,5)[b]}
\put(65,1){\oval(10,5)[b]}
 
 \qbezier[20](45,-5),(45,0),(45,10)
\end{picture} .

\begin{prop}   \label{prop:stabilises}
{\sf (i)} \ For a fixed $m\in\BN$, the numbers $\boldsymbol{F}_{n,n-m}$ stabilise for $n\ge 2m+1$. In other words,
$\boldsymbol{F}_{n,n-m}=\boldsymbol{F}_{2m+1,m+1}$ for all $n\ge 2m+1$. 
\\ 
{\sf (ii)} \ Furthermore, $\boldsymbol{F}_{2m+1,m+1}=\boldsymbol{F}_{2m,m}+1$.
\end{prop}
\begin{proof}
(i) \ Let  $\q=\q^{\sf C}_n(\un{a}{\mid}\un{b})\in \boldsymbol{\eus F}_{n,n-m}$. Then $\sum_{i=1}^s a_i=m$ 
and $\sum_{j=1}^t b_j=n$.
Consider the $n$-th vertex  of the graph (one that is closest to the $\sigma$-mirror). We are interested in 
$b_t$, the size of the last part of $\un{b}$, i.e., the part that contains the $n$-th vertex.
By the assumption, we have $n-m$ central arcs over the horizontal line. Therefore, 
if $n\ge 2m+2$ and $b_t\ge 2$, then the smallest arc corresponding to $b_t$ hits two vertices covered by 
central arcs above the line. And this produces a cycle in the graph! This contradiction shows that the 
only possibility is $b_t=1$. Then one can safely remove two central vertices from the graph and 
conclude that $\boldsymbol{F}_{n,n-m}=\boldsymbol{F}_{n-1,n-1-m}$ as long as $n\ge 2m+2$.
(The last step is opposite to one that is used in the proof of Lemma~\ref{lm:monotonicity}.)
\\ \indent
(ii) \ Again, for $\q=\q^{\sf C}_{2m+1}(\un{a}{\mid}\un{b})\in \boldsymbol{\eus F}_{2m+1,m+1}$, we consider $b_t$, 
the last coordinate of $\un{b}$. If $b_t=1$, then the central pair of vertices in $\Gamma^{\sf C}(\q)$ can 
be removed, which yields a seaweed in $\boldsymbol{\eus F}_{2m,m}$. Next, it is easily seen that if 
$b_t\in \{2,3,\dots,2m\}$, then $\Gamma^{\sf C}(\q)$ contains a cycle. Hence this is impossible. While 
for $b_t=2m+1$, one obtains a unique admissible possibility $\un{a}=(\underbrace{1,1,\dots,1}_{m})$. 
\end{proof}

\begin{rema}   Using a similar analysis, one can  show that 
$\boldsymbol{F}_{2m,m}=\boldsymbol{F}_{2m-1,m-1}+3$, if $m\ge 3$.
\end{rema}

\begin{rmk}
Our stabilisation result for $\boldsymbol{F}_{n,n-m}$ can be compared with~\cite{duyu}, where Duflo and Yu consider a partition of 
the set of standard Frobenius seaweeds in $\sln$ into classes and study the asymptotic behaviour of 
the cardinality of these classes as $n$ tends to infinity.  Let $p(\un{a})$ be the number of nonzero parts 
of the composition $\un{a}$ and let $\tilde F_{n,p}$ be the number of the standard Frobenius seaweeds 
$\q^{\sf A}(\un{a}{\mid}\un{b})\cap\sln$ such that $p(\un{a})+p(\un{b})=p$. By~\cite[Theorem\,1.1(b)]{duyu}, if $n$ is sufficiently large, then $\tilde F_{n,n+1-t}$ is a polynomial in $n$ of degree
$[t/2]$, with positive rational coefficients.
\end{rmk}

It seems that $\boldsymbol{\eus F}_{n,1}$ is the most interesting part of the symplectic Frobenius 
seaweeds. Recall from Section~\ref{sect:generalities} that to any standard seaweed $\q\subset\spn$ 
one can associate a ``symmetric'' seaweed $\tilde\q\subset \gltn$ such that $\q=\tilde\q\cap\spn$. In this 
context, we also set $\tilde\q_0=\tilde\q\cap\sltn$.

\begin{prop}   \label{prop:f_(n,1)}
{\sf (i)} \  If\/ $\q\in \boldsymbol{\eus F}_{n,1}$, then $\ind\tilde\q=1$, hence $\tilde\q_0$ is a Frobenius 
seaweed in $\sltn$.
\\ 
{\sf (ii)} \  There is an injective map $\boldsymbol{\eus F}_{n,1}\to \boldsymbol{\eus F}_{n+1,1}$, 
which is not onto if $n\ge 2$.
\end{prop}
\begin{proof}
(i) \  If $\q\in\boldsymbol{\eus F}_{n,1}$, then $\Gamma^{\sf C}(\q)$ 
and thereby $\Gamma^{\sf A}(\tilde\q)$ consists of a sole segment (Lemma~\ref{lm:k-connected}). By Eq.~\eqref{eq:index-A}, we have $\ind\tilde\q=1$ and therefore $\ind\tilde\q_0=\ind\tilde\q-1=0$.
\par (ii) \ 
If $\q=\q^{\sf C}_{n}(\un{a}{\mid}\un{b})\in \boldsymbol{\eus F}_{n,1}$, then $\sum_{i=1}^s a_i=n-1$ and 
$\sum_{j=1}^t b_j=n$. We associate to it a seaweed $\hat\q\in\boldsymbol{\eus F}_{n+1,1}$ as follows.
Set $\hat\q=\q^{\sf C}_{n+1}(\un{\hat a}{\mid}\un{b})$, where $\un{\hat a}=(a_1,\dots,a_s,2)$.
Note that $\Gamma^{\sf C}_{n}(\un{a}{\mid}\un{b})$ has one central arc above the horizontal line,
while $\Gamma^{\sf C}_{n+1}(\un{\hat a}{\mid}\un{b})$ has one central arc below. The following  is a 
graphical illustration of the transform $\q \mapsto \hat\q$:

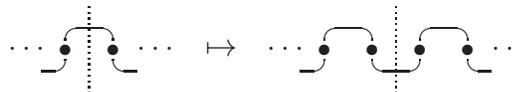
\begin{center}
\begin{picture}(60,20)(-10,3)
\setlength{\unitlength}{0.025in}
\multiput(0,3)(10,0){2}{\circle*{2}}

\put(-12,3){$\dots$}
\put(15,3){$\dots$}
\put(5,5){\oval(10,5)[t]}

\put(15,1){\oval(10,5)[lb]}
\put(-5,1){\oval(10,5)[rb]}

\qbezier[15](5,-6),(5,3),(5,12)
\end{picture}   
$\mapsto$
\begin{picture}(50,20)(-30,3)
\setlength{\unitlength}{0.025in}
\multiput(0,3)(10,0){4}{\circle*{2}}

\put(-12,3){$\dots$}
\put(35,3){$\dots$}
\put(5,5){\oval(10,5)[t]}
\put(25,5){\oval(10,5)[t]}

\put(15,1){\oval(10,5)[b]}
\put(35,1){\oval(10,5)[lb]}
\put(-5,1){\oval(10,5)[rb]}

\qbezier[15](15,-6),(15,3),(15,12)
\end{picture} 
\end{center}
\vskip1ex
This provides a bijection between $\boldsymbol{\eus F}_{n,1}$ and the seaweeds in
$\boldsymbol{\eus F}_{n+1,1}$ whose last part of the composition that sums up to $n+1$ equals 2.
If $n+1\ge 3$, then there are seaweeds in $\boldsymbol{\eus F}_{n+1,1}$ such that the above-mentioned last part is bigger than 2. Hence $\boldsymbol{F}_{n,1}<\boldsymbol{F}_{n+1,1}$.
\end{proof}

\begin{rmk}  \label{rem:frobC&A}
Another curious observation is that $\boldsymbol{\eus F}_{n,1}$ and $\boldsymbol{\eus F}_{n,2}$ are related 
to certain Frobenius seaweeds in $\sln$: 
\par (i) Suppose that $\q\in\boldsymbol{\eus F}_{n,1}$. Let us remove the only central arc in $\Gamma^{\sf C}(\q)$ and take the remaining left hand half of the graph as it is. It is a {\bf connected} type-{\sf A} meander graph with $n$ vertices. Therefore, it represents a seaweed of index $1$ in $\gln$ 
(=\,Frobenius seaweed in $\sln$). Formally, if $\q=\q^{\sf C}_n(\un{a}{\mid}\un{b})$, with $\sum a_i=n-1$ and $\sum b_j=n$, then we set $\q'=\q^{\sf A}(\un{a}'{\mid}\un{b})\subset\sln$, where
$\un{a}'=(\un{a},1)$. This yields a bijection between $\boldsymbol{\eus F}_{n,1}$ and the Frobenius seaweeds of
$\sln$ such that the last part of $\un{a}'$ equals $1$.
\par (ii) Suppose that $\q\in\boldsymbol{\eus F}_{n,2}$. Let us remove the two central arcs and 
take the remaining left hand half. We obtain a graph with $n$ vertices and two connected components
(segments). Joining the last two ``lonely'' vertices by an arc, we get a {\bf connected} type-{\sf A} meander 
graph. Formally, if 
$\q=\q^{\sf C}_n(\un{a}{\mid}\un{b})$, with $\sum a_i=n-2$ and $\sum b_j=n$, then we set 
$\q'=\q^{\sf A}(\un{a}'{\mid}\un{b})\subset\sln$, where $\un{a}'=(\un{a},2)$. Again, this yields 
a bijection between $\boldsymbol{\eus F}_{n,2}$ and 
the Frobenius seaweeds of $\sln$ such that the last part of $\un{a}'$ equals $2$.

Unfortunately, such a nice relationship does not extend to $\boldsymbol{\eus F}_{n,3}$.
\end{rmk}

We present the table of numbers $\boldsymbol{F}_{n,k}$ for $n\le 7$.
\begin{table}[htb]
\begin{center}
\begin{tabular}{c|ccccccc||c}
\begin{picture}(15,15)
\setlength{\unitlength}{0.015in}
\put(0,2){\small $n$}
\put(10,8){\small $k$}
\put(1,14){\line(1,-1){16}}
\end{picture}
& 1& 2& 3& 4& 5& 6 & 7 & $\Sigma=\boldsymbol{F}_n$\\ \hline
1 & 1 & -& -& -& -& - & - &1\\
2 & 1 & 1 & -& -& -& - & - & 2\\
3 & 2 & 2 & 1 & -& -& - & - &5 \\
4 & 4 & 4 & 2 & 1 & -& - & - &11\\
5 & 8 & 10 & 5 & 2 & 1 & - & - &26\\
6 & 15 & 20 & 13 & 5 & 2& 1 &- & 56\\
7 & {28} & {44} & 28 & 14 & 5 & 2 & 1 &  {122} \\ \hline
\end{tabular}
\end{center}
\vskip1ex

\caption{The numbers $\boldsymbol{F}_{n,k}$ for $n\le 7$}   \label{table-1}
\end{table}

\noindent
Note that the values $14,5,2,1$ in the 7-th row are stable in the sense of Proposition~\ref{prop:stabilises}(i).
Using preceding information, we can also compute the next stable value:
\[
   \boldsymbol{F}_{9,5}=\boldsymbol{F}_{8,4}+1=(\boldsymbol{F}_{7,3}+3)+1=32 .
\]

\section{On meander graphs for the odd orthogonal Lie algebras}
\label{sect:odd-orth}

\noindent
As in the case of $\spn$, the standard parabolic subalgebras of $\sono$ are parametrised
by the compositions $\un{a}=(a_1,\dots,a_s)$ such that $\sum a_i\le n$. For instance, if
$\p^{\sf B}_n(\un{a})$ is the standard parabolic subalgebra corresponding to $\un{a}$,
then a Levi subalgebra of it is of the form $\gt{gl}_{a_1}{\oplus}\dots\oplus\gt{gl}_{a_s}\oplus
\mathfrak{so}_{2(n-\sum a_i)+1}$. Therefore,
the standard seaweed subalgebras of $\sono$ are also parametrised
by the pairs of compositions $\un{a}$, $\un{b}$ such that $\sum a_i\le n$ and $\sum b_j\le n$,
see~\cite[Section\,5]{Dima01}.  Furthermore, the inductive procedure for computing the index of standard
seaweeds (see Section~\ref{sect:main}, Step {\sl\bfseries 2.}), which reduces the case of arbitrary seaweeds to parabolic subalgebras, also remains the same~\cite[Theorem\,5.2]{Dima01}.

This means that if the formula for the index of parabolic subalgebras of $\sono$ in terms of
$\un{a}$ also remains the "same" as in the symplectic case, then one can use our type-{\sf C} meander 
graphs in type $\GR{B}{n}$ as well. Although, there are only partial results on the index of parabolic
subalgebras of $\sono$ in \cite[Section\,6]{Dima01}, one can use the  general Tauvel-Yu-Joseph formula, 
see~\cite[Conj.\,4.7]{ty-AIF} and \cite[Section\,8]{jos}. Namely, if $\q=\q(S,T)$ is the seaweed corresponding to the subsets $S,T\subset \Pi$, then 
\beq    \label{eq:tau-yu}
   \ind\q=\rk \g +\dim E_S +\dim E_T  -2\dim (E_{S}+E_T) .
\eeq
Here $\dim E_T=\# {\mathcal K}(T)$ is the cardinality of the cascade of strongly orthogonal roots
in the Levi subalgebra of $\g$ corresponding to $T$, see \cite{ty-AIF} for the details.
Our observation is that it easily implies that, for any composition $\un{a}$, one has
\beq    \label{eq:ind-parab-B}
    \ind \p^{\sf B}_n(\un{a})=\left[ \frac{a_1}{2} \right ]+\ldots+\left[ \frac{a_s}{2} \right ]+(n-\sum_{i=1}^s a_i)
    =\ind\p^{\sf C}_n(\un{a}) .
\eeq
Indeed, for the parabolic subalgebras, we may assume that $S=\Pi$, and since $\ind\be=0$ for the series $\GR{B}{n}$, we have $\dim E_\Pi=\rk\g$. Therefore,  
$\ind \p^{\sf B}_n(\un{a})=\dim E_T=\# {\mathcal K}(T)$. 
As already noticed before, for  
$\p^{\sf B}_n(\un{a})$, we have $\el=\gt{gl}_{a_1}{\oplus}\dots\oplus\gt{gl}_{a_s}\oplus
\mathfrak{so}_{2(n-\sum a_i)+1}$. As is well-known, the cardinality of the cascade of strongly orthogonal roots in  $\mathfrak{gl}_a$ (resp. $\mathfrak{so}_{2n+1}$) equals $[a/2]$ (resp. $n$), see \cite[Sect.\,2]{jos77}. Therefore, the cardinality of the cascade in the above $\el$ is given by the middle term 
in~\eqref{eq:ind-parab-B}.

There is another interesting formula for the index of a parabolic subalgebra, 
which generalises the above observation.   

\begin{thm}    \label{thm:ind-parab-gen}
Let $\g$ be a simple Lie algebra such that $\ind \be=0$. Let $\p\subset\g$ be a parabolic subalgebra,
with a Levi subalgebra $\el$. If\/ $\be(\el)$ is a Borel subalgebra of\/ $\el$ and
$\ut(\el)=[\be(\el),\be(\el)]$,  then 
\beq    \label{eq:teor}
\ind \p=\ind \ut(\el)  =\rk\el-\ind \be(\el) =\rk\g-\ind \be(\el) .
\eeq
In particular, $\ind\p=0$ if and only if\/ $\ut(\el)=0$, i.e., $\p=\be$.
\end{thm}

\noindent
{\it\bfseries Outline of the proof.}
Again, under the assumption that $\ind\be=0$, we have $S=\Pi$, $\dim E_\Pi=\rk\g$, and $\el$ 
is determined by $T$. Hence \eqref{eq:tau-yu} implies that
$\ind\p=\dim E_T=\# {\mathcal K}(T)$. It is implicit in \cite[2.6]{jos77} that 
$\# {\mathcal K}(T)=\ind \ut(\el)$,
and the second 
equality in \eqref{eq:teor} is a consequence of the fact that $\rk\el=\ind \be(\el)+\ind \ut(\el) $ for any reductive Lie algebra $\el$. 
A more detailed explanation  and some applications of the theorem will appear elsewhere. \hfil \qed

Recall that, for a simple Lie algebra $\g$, $\ind\be=0$ if and only if $\g\ne \GR{A}{n},\GR{D}{2n+1},
\GR{E}{6}$.  

{\bf Conclusion}. {\sf 1)} \ Given a standard seaweed $\q=\q^{\sf B}_n(\un{a}{\mid}\un{b})\subset\sono$, we 
can draw exactly the same meander graph as in type {\sf C} (with $2n$ vertices) 
and use exactly the same topological formula (Theorem~\ref{thm:main}) to compute the index of $\q$.

{\sf 2)} \ Using our type-{\sf C} meander graphs, we can establish a bijection between the standard Frobenius seaweeds for the symplectic and odd orthogonal Lie algebras of the same rank. It would be very interesting to realise whether there is a deeper reason for such a bijection.

{\sf 3)} \  For the even-dimensional orthogonal Lie algebras (type $\GR{D}{n}$), there is a similar inductive procedure that reduces the problem of computing the index of arbitrary seaweeds to parabolic subalgebras.
However, $\ind\be=1$ for $\GR{D}{2n+1}$ and Theorem~\ref{thm:ind-parab-gen} does not apply.
Furthermore, although $\ind\be=0$ for $\GR{D}{2n}$,  
the general formula for the index of parabolic subalgebras cannot be expressed nicely in terms of compositions. Of course, the reason is that the Dynkin diagram has a branching node!

\end{document}